\documentclass[11pt,leqno]{amsart}

\usepackage{graphicx}
\usepackage{amssymb}
\usepackage{color}
\newtheorem{thm}{\bf Theorem}
\newtheorem{lem}[thm]{\bf Lemma}
\newtheorem{prop}[thm]{\bf Proposition}

\newcommand{\lc}{locally finite}
\newcommand{\X}{\mathcal{X}}
\newcommand{\Y}{\mathcal{Y}}
\newcommand{\ol}{\overline}

\begin{document}
\author{Eloisa Detomi}
\address{Dipartimento di Matematica, Universit\`a di Padova, 
 Via Trieste 63, 35121 Padova , Italy}
\email{detomi@math.unipd.it}
\author{Marta Morigi}
\address{Dipartimento di Matematica, Universit\`a di Bologna,
Piazza di Porta San Donato 5, 40126 Bologna, Italy}
\email{mmorigi@dm.unibo.it}
\author{Pavel Shumyatsky}
\address{Department of Mathematics, University of Brasilia,
Brasilia-DF, 70910-900 Brazil}
\email{pavel@unb.br}

\title[Commutators in profinite groups]{Commutators and pronilpotent subgroups in profinite groups}
 \thanks{ The research of the first two authors was partially supported by MIUR and GNSAGA. 
 The research of P. Shumyatsky was  supported by CAPES and CNPq.}
\begin{abstract}
Let $G$ be a profinite group in which all pronilpotent subgroups generated by commutators are periodic. 
We prove that $G' $ is locally finite.

\end{abstract}

\maketitle

\section{Introduction} An element of a group $G$ that can be written in the form $x^{-1}y^{-1}xy$ for suitable $x,y\in G$ is called commutator. A number of outstanding results about commutators in finite and profinite groups have been obtained in recent years. In this context we mention the proof by Liebeck, O'Brien, Shalev and Tiep \cite{lost} of Ore's conjecture: Every element of a finite simple group is a commutator. Another significant result is due to Nikolov and Segal, who proved that if $G$ is an $m$-generated finite group, then every element of $G'$ is a product of $m$-boundedly many commutators \cite{nisegal}. Here, as usual, $G'$ denotes the derived subgroup of $G$, that is, the subgroup generated by all commutators. One can ask a general question -- given a group $G$ with certain specific restrictions on commutators, what kind of information on the structure of $G'$ one can deduce?

Let $G$ be a finite group and $P$ a Sylow $p$-subgroup of $G$. A corollary of the Focal Subgroup Theorem \cite[Theorem 7.3.4]{go} is that $P\cap G'$ is generated by commutators. From this one immediately deduces that if all nilpotent subgroups generated by commutators have exponent dividing $e$, then the exponent of the derived group $G'$ divides $e$, too. The routine inverse limit argument allows to easily obtain a related fact for profinite groups: if $G$ is a profinite group in which all pro-$p$ subgroups generated by commutators have finite exponent dividing $e$, then the exponent of the derived group $G'$ divides $e$. Recall that a group is said to be of finite exponent $e$ if $x^e=1$ for each $x\in G$.  A profinite group is a topological group that is isomorphic to an inverse limit of finite groups. The textbooks \cite{ribes-zal} and \cite{wilson} provide a good introduction to the theory of profinite groups. In the context of profinite groups all the usual concepts of group theory are interpreted topologically. In particular, by a subgroup of a profinite group we mean a closed subgroup. A subgroup is said to be generated by a set $S$ if it is topologically generated by $S$.

The purpose of the present paper is to prove the following theorem.
\begin{thm}\label{main}
Let $G$ be a profinite group in which all pronilpotent subgroups generated by commutators are periodic. Then $G' $ is locally finite. 
\end{thm}

Recall that a group is periodic (torsion) if every element of the group has finite order. A group is locally finite if each of its finitely generated subgroups is finite. Periodic profinite groups have received a good deal of attention in the past. In particular, using Wilson's reduction theorem \cite{wilson-torsion}, Zelmanov has been able to prove local finiteness of periodic profinite groups \cite{z}. Earlier Herfort showed that there exist only finitely many primes dividing the orders of elements of a periodic profinite group \cite{he}. It is a long-standing problem whether any periodic profinite group has finite exponent.

\section{Preliminary lemmas} We will slightly abuse terminology and use the term ``prosoluble group" to mean a topological group that is isomorphic to an inverse limit of finite soluble groups. 

Let us denote by $\X$ the class of profinite groups with the property that every pronilpotent subgroup generated by commutators is periodic. Moreover, let us denote by $\Y$ the class of profinite groups in which every commutator has finite order and the derived subgroup of every Sylow subgroup is periodic. It is clear that the class $\Y$ is closed under taking subgroups and homomorphic images. We also remark that $\X \subseteq \Y$.

If $G$ is a finite soluble group, $h(G)$ denotes the Fitting height of $G$. Recall that this is the length of the shortest normal (or characteristic) series of $G$ all of whose quotients are nilpotent. If $G$ is a prosoluble group having a finite series of characteristic subgroups all of whose quotients are pronilpotent, we denote the shortest length of such a series by $ph(G)$. If $G$ has no such series, we write $ph(G)=\infty$. Of course $ph(G)<\infty$ if and only if $G$ is an inverse limit of finite soluble groups of bounded Fitting height. 

As usual, if $G$ is a finite group, $\pi(G)$ is the set of primes dividing the order of $G$. For a profinite group $G$ we denote by $\pi(G)$ the (possibly infinite) set $\pi$ of primes such that $G$ is an inverse limit of finite groups of orders divisible by primes in $\pi$ only. In what follows we use the Wilson-Zelmanov Theorem that every periodic profinite group is locally finite without explicit references.

An important role in the present paper will be played by the following two theorems, due to Wilson \cite[Theorems 2 and 3]{wilson-torsion}: 
\begin{thm}\label{wil-series} 
Let $p$ be a prime and $G$ a profinite group whose Sylow $p$-subgroups are periodic. Then $G$ has a finite series of closed
characteristic subgroups in which each factor either is pro-($p$-soluble)
or is isomorphic to a Cartesian product of non-abelian finite simple
groups of order divisible by $p$.
\end{thm}

\begin{thm}\label{wilso-series} 
Let $p$ be an odd prime and G a pro-($p$-soluble) group whose Sylow $p$-subgroups are periodic. Then $G$ has a finite series of closed characteristic subgroups in which each factor is either a pro-$p$ group or a pro-$p'$ group.\end{thm}

One immediate corollary of Theorem \ref{wilso-series} is that if $\pi$ is a finite set of odd primes and G a prosoluble group whose Sylow $p$-subgroups are periodic for every $p\in\pi$, then $G$ has a finite series of closed characteristic subgroups in which each factor is either a pro-$\pi$ group or a pro-$\pi'$ group. 

Another corollary is that if $G$ is a prosoluble group for which $\pi(G)$ is finite and all of whose Sylow subgroups are periodic, then $ph(G)$ is finite.

\begin{lem}\label{G'} 
Let $G\in\X$. Then all Sylow $p$-subgroups of $G'$ are periodic.
\end{lem}
\begin{proof} If $p$ is a prime, a  Sylow $p$-subgroup of $G'$ is of the form
$P \cap G'$, for some Sylow $p$-subgroup $P$ of $G$. A profinite version of the Focal Subgroup Theorem (see Theorem 4.4 of \cite{gilotti-ribes-serena}) implies that $P\cap G'$ is generated by commutators. Since $G\in\X$, it follows that $P\cap G'$ is periodic.
\end{proof}

\begin{lem}\label{finite-pi} 
Let $G$ be a  profinite group all of whose Sylow subgroups are periodic. If $\pi(G)$ is finite, then $G$ is locally finite.
\end{lem}
\begin{proof} It is sufficient to show that $G$ is periodic. Choose $x\in G$ and consider the procyclic subgroup $\langle x\rangle$ generated by $x$. Write $\langle x\rangle=S_1S_2\dots S_n$, where $S_i$ are the Sylow subgroups of $\langle x\rangle$. Since all the subgroups $S_i$ are finite cyclic and since there are only finitely many primes in $\pi(G)$, we conclude that $\langle x\rangle$ is finite. This completes the proof.
\end{proof}

\begin{lem}\label{baba} Let $G$ be a finite soluble group whose Hall $2'$-subgroups are abelian. Then $h(G)\leq 3$.
\end{lem}
\begin{proof} It is sufficient to show that if $G$ has no non-trivial normal 2-subgroups, then a Hall $2'$-subgroup of $G$ is normal. So assume that $G$ has no non-trivial normal 2-subgroups and choose a Hall $2'$-subgroup $H$ in $G$. Let $F$ be the Fitting subgroup of $G$. Then $F$ is a $2'$-subgroup and therefore $F\leq H$. Since $C_G(F)\leq F$ \cite[Theorem 6.1.3]{go} and since $H$ is abelian, we conclude that $F=H$. The lemma follows.
\end{proof}

Let $G$ be a finite group acted on by a finite group $A$ such that $(|G|,|A|)=1$. It is well-known that if $N$ is a normal $A$-invariant subgroup of $G$, then $C_{G/N}(A)=C_G(A)N/N$ (see for example \cite[6.2.2]{go}). A profinite version of this result can be stated as follows. 

Let $G$ be a profinite group on which a finite group $A$ acts by continuous  automorphisms. Suppose that $\pi(G)\cap\pi(A)=\emptyset$. If $N$ is a closed normal $A$-invariant subgroup of $G$, then $C_{G/N}(A)=C_G(A)N/N$. A proof of this can be found in \cite{io98} (see also Lemma 2 in \cite{he} for the case where $|A|$ is a prime).

In his seminal work \cite{th} Thompson showed that if $G$ is a finite soluble group acted coprimely by a finite group $A$, then $h(G)$ is bounded in terms of $h(C_G(A))$ and the number of prime divisors of $|A|$, counting multiplicities. A survey of further results in this direction can be found in Turull \cite{turull}. We will require a profinite version of Thompson's theorem. It can be deduced by the standard inverse limit argument using the above mentioned property 
 that if $N$ is a closed normal $A$-invariant subgroup of $G$, then $C_{G/N}(A)=C_G(A)N/N$.  Therefore we omit the proof.

\begin{prop}\label{thompson} Let $G$ be a  prosoluble group acted on by a finite group $A$ and suppose that $\pi(G)$ and  $\pi(A)$ are disjoint. If $C_G(A)$ has a finite normal series with pronilpotent quotients, then also $G$ has such a series.
\end{prop}

\begin{lem}\label{ab}
Let $G$ be a non-abelian prosoluble group all of whose Sylow subgroups are periodic.
Then  there exists a subgroup of $G$ which is non-abelian and finite. 
\end{lem}
\begin{proof} Since $G$ is non-abelian, we can choose two not necessarily distinct primes $p$ and $q$ such that a Hall $\{p,q\}$-subgroup $M$ is  non-abelian. Since by Lemma \ref{finite-pi} $M$ is locally finite, the result follows.
\end{proof}

\section{Proof of the main theorem}
\begin{prop}\label{pi-prosol} 
Let $G$ be a prosoluble group all of whose Sylow subgroups are periodic and assume that every commutator in $G$ has finite order. Then $ph(G)$ is finite. 
\end{prop}
\begin{proof} Assume that $ph(G)$ is infinite. In what follows we will inductively construct a sequence of  subgroups $A_1, A_2, \ldots $, 
 an increasing chain of normal closed subgroups $$1=R_1 \le R_2 \le \ldots,$$ and a decreasing chain of normal open subgroups $$G=H_1 \ge H_2 \ge \ldots\ge H=\cap H_i $$ such that
\begin{enumerate} 
\item there exist mutually disjoint sets of  odd primes $\pi_i$ with the property that the image of $A_i$ in $G/R_i$  is a finite non-abelian $\pi_i$-group for every $i$; 
\item $[A_i,A_j]\le R_j$ whenever $i<j$;
\item $ph(G/R_i)=\infty$ for every $i$;
\item $A_i\cap H_{i+1}=R_i$.
\end{enumerate}
Later we will see that the above properties lead to a contradiction.

Since $ph(G)$ is infinite, Lemma \ref{baba} implies that the Hall $2'$-subgroups of $G$ are non-abelian. 
By Lemma \ref{ab} we can find
 a non-abelian finite subgroup 
$A_1$ of odd order. Let $\pi_1=\pi(A_1)$ and set $R_1=1$. Next, choose a normal open subgroup $H_2$ in $G$ 
such that $A_1\cap H_2=1$. 
 As $\pi_1$ is a set of odd primes it follows from Theorem \ref{wilso-series} 
 that $H_2$ has a finite series of closed characteristic 
subgroups $\{N_{i}\}$ in which each quotient $ N_{i}/ N_{i+1}$ is either a pro-$\pi_1$ or a pro-$\pi_1'$ group. 
Since $\pi_1$ is finite, $ph(N_{i}/ N_{i+1})$ is finite whenever the quotient $N_{i}/ N_{i+1}$ is pro-$\pi_1$. 
Since $ph(G)=\infty$, we have $ph(H_2)=\infty$. Therefore there exists a  pro-$\pi_1'$ quotient $N_{i}/N_{i+1}$ 
such that $ph(N_{i}/N_{i+1})=\infty$. Set $R_2=N_{i+1}$. We will now pass to the quotient $G/R_2$.

The finite $\pi_1$-group $A_1$ acts in the natural way on the $\pi_1'$-group 
$N_i/R_2$. By Proposition \ref{thompson} $ph(C_{N_i/R_2}(A_1))=\infty$. Lemma \ref{baba} shows that Hall 
$2'$-subgroups in $C_{N_i/R_2}(A_1)$ are non-abelian. Thus, by Lemma \ref{ab} there exists a subgroup $A_2$ such that 
  $A_2/R_2$ is a finite   non-abelian subgroup  of odd order in  $C_{N_i/R_2}(A_1)$. 
Set  $\pi_2=\pi(A_2/R_2)$.
 We note that $\pi_1\cap\pi_2=\emptyset$ and $[A_1,A_2]\leq R_2$.

We now repeat the above arguments in the following way. Assume by induction that
 the subgroups $A_1, \ldots , A_{n-1}$ 
 and the  normal subgroups $R_1,\ldots,R_{n-1}$ and 
$H_1,\ldots ,H_{n-1}$ with the prescribed properties have been found. For brevity we indicate with an 
overline the elements (or subgroups) modulo the subgroup $R_{n-1}$. Then the $\ol A_1, \ldots ,\ol A_{n-1}$ 
are pairwise commuting finite subgroups of mutually coprime orders. In particular, 
$\ol A=\ol A_1 \cdots \ol A_{n-1}$ is a finite $\pi$-group, where $\pi=\cup_{i=1}^{n-1}\pi_i$. 
Let $J$ be a normal open subgroup of $G$ with the properties that $R_{n-1} \le J$ and  $\ol J \cap \ol A=1$. 
Let  $H_n=H_{n-1}\cap J$.  Note that $\pi$ consists of odd primes, so by  
Theorem \ref{wilso-series} the subgroup $\ol H_n$ has a finite series of closed characteristic subgroups $\{\ol N_{i}\}$ in which each quotient $\ol N_{i}/\ol N_{i+1}$ is either a pro-$\pi$ or a pro-$\pi'$ group. Since $\pi$ is finite, $ph(\ol N_{i}/\ol N_{i+1})$ is finite for each pro-$\pi$ quotient $\ol N_{i}/\ol N_{i+1}$. Taking into account that $ph(\ol G)= \infty$, we remark that also $ph(\ol H_n)= \infty$. Therefore there exists a pro-$\pi'$ quotient $T_n=\ol N_{i}/\ol N_{i+1}$ for some index $i$, such that $ph(T_n)=\infty$.
 Let  
 $R_n/R_{n-1}=\ol N_{i+1}$. 

The finite $\pi$-group $\ol A$ naturally acts on the $\pi'$-group $T_n$. By Proposition \ref{thompson},  
$ph(C_{T_n}(\ol A))=\infty$. Lemma \ref{baba} shows that Hall $2'$-subgroups in $C_{T_n}(\ol A)$ are non-abelian.
We therefore can use Lemma \ref{ab} and find a subgroup $A_n$ such that 
 \begin{enumerate} 
\item[(a)]  $A_n/R_n \le C_{T_n}(\ol A)$ 
\item[(b)] $A_n/R_n$ is a finite non-abelian $\pi'$-group of odd order.
\end{enumerate}

 It is now easy to see that the required properties are satisfied by
$A_1,\dots,A_n$, $R_1,\dots,R_n$ and $H_1,\dots,H_n$. This completes the inductive step. Hence we now assume 
that the non-abelian finite subgroups $A_i$, an increasing chain of normal closed subgroups 
$1=R_1 \le R_2 \le \ldots$, and a decreasing chain of normal open subgroups $$G=H_1 \ge H_2 \ge \ldots\ge H=\cap H_i$$
 with  the prescribed properties have been found.

Remark that $H$ contains $\cup_{n=1}^{\infty} R_n$. Since $H$ is closed, we can consider the quotient $G/H$ of 
$G$. The image in $G/H$ of the closed subgroup generated by all $A_i$ is isomorphic with the 
Cartesian product of finite non-abelian groups of mutually coprime orders. Obviously this group contains a 
commutator of infinite order. This yields a contradiction since all commutators in $G$ have finite order. 
The proof is now complete.
 \end{proof}

\begin{lem}\label{Nabel}
Let $N$ be an  abelian normal subgroup of a group $G \in\Y$.
 Then there exists an open normal subgroup $H$ of $G$ 
such that $[N,H]$ is \lc. 
\end{lem}
\begin{proof}
Let $a$ be an element of $G$. As $N$ is abelian, the subgroup 
$[N,a]$ coincides with the set of commutators $\{[n,a]\,|n\, \in N\}$, and therefore it is a 
(closed) periodic group. In particular $[N,a]$ has finite exponent 
(see e.g. \cite[Lemma 4.3.7]{ribes-zal}). 

For every integer $m$ consider 
\[ S_m = \{ a\in G \mid [N,a]^m=1 \}. \] Note that the sets $ S_m$ are closed and they cover $G$. By Baire's Category Theorem \cite[p.\ 200]{ke} there exist an integer $m$, an open normal subgroup $H\leq G$ and an element $a\in G$ such that 
\[ [N,aH]^m=1. \]
Since $[N,a]^G \le N$ is abelian and  generated by elements of orders dividing $m$, it has exponent dividing $m$. In particular $[N,a]^G$ is \lc. We can now pass to the quotient $G/[N,a]^G$. Without loss of generality we  assume that $[N,a]=1$. For every $n\in N$ and $h\in H$ we have $[n,ah]=[n,h]$. Therefore $[N,H]$ is generated by elements of orders dividing $m$. Thus $[N,H]$, being abelian, has exponent dividing $m$, and so it is \lc, as required. 
\end{proof}

\begin{lem}\label{abelian}
Assume that $B$ is a profinite group in which every commutator has finite order, and let $A$ be an open normal abelian subgroup of $B$. Then $[A,B]$ is locally finite.
\end{lem}
\begin{proof}
Let $b_1, \dots, b_k$ be a set of representatives of the left cosets of $A$ in $B$. Then
\[[A, B] = \prod_{i=1}^{k} [A,b_i].\]
 By the same argument as in Lemma \ref{Nabel} each $[A,b_i]$ is locally finite, thus $[A, B] $ is 
also \lc.
\end{proof}

\begin{prop}\label{prosol+h=>lc}
Let $G \in \Y$ be prosoluble and assume that $ph(G)$ finite. Then $G'$ is \lc.  
\end{prop}
\begin{proof}
Recall that each subgroup and   each quotient of $G$ is in $\Y$. We prove the proposition by induction on  
$h=ph(G)$. If  $h=1$, then $G$ is pronilpotent and $G'$ is the 
Cartesian product of the derived subgroups $G'_p$ of  the Sylow subgroups $G_p$ of $G$. As $G\in\Y$, each $G'_p$ 
is \lc. Moreover $\pi(G')$ is finite, because otherwise $G$ would contain a commutator with infinite order, a contradiction. Therefore $G'$ is \lc, as required. 

Assume now that $h\ge 2$. The group $G$ has a characteristic series with
pronilpotent quotients  of length $h$. Let $N\le K$ be the last two terms of  the series. Thus, $K/N$ and $N$ 
are pronilpotent. Since $N$ is pronilpotent, our argument on the case $h=1$ yields that $N'$ is \lc. Factoring $N'$ out, we can assume that $N$ is abelian. By Lemma \ref{Nabel} there exists an open normal subgroup $H$ of $K$ such that $[N,H]$ is  \lc. Clearly we can assume that $H$ is normal in $G$ and pass to the quotient $G/[N,H]$. Thus, we can assume that $[N,H]=1$. 
 
Then $N \cap H$ is contained in the centre of $H$. As $H/N\cap H\cong HN/N \le K/N$ is pronilpotent, we conclude that $H$ is pronilpotent. 
Since in the case where $h=1$ the lemma holds, we conclude that $H'$ is \lc. We now pass to the quotient $G/H'$ and assume that $H'=1$.

In this case $H$ is an abelian subgroup of finite index in $K$. By Lemma \ref{abelian} $[H, K]$ is \lc. 
Factoring $[H,K]$ out, we obtain that $H$ is contained in the centre of $K$. Hence $K$ is nilpotent. 
By induction on $h$ the proposition follows. 
\end{proof}

\begin{lem}\label{Y-nonab} Let $H \in \Y$ be a Cartesian product of non-abelian finite simple groups. 
Then $H$ has finite exponent. 
\end{lem}
\begin{proof}
Let $H$ be a Cartesian product of the non-abelian finite simple groups $S_i$, $i\in I$. If $H$ has infinite 
exponent, then for each 
$n \in \mathbb N$ there exists an index $i_n$ and an element $s_{i_n} \in S_{i_n}$ such that $|s_{i_n}| >n$. 
By the positive solution of Ore's Conjecture \cite{lost} each $s_{i_n}$ is a commutator, 
say $s_{i_n}=[x_{i_n},y_{i_n}]$. Then the element 
\[\left[\prod_{n \in \mathbb N} x_{i_n}, \prod_{n \in \mathbb N} y_{i_n}\right]= \prod_{n \in \mathbb N} [x_{i_n}, y_{i_n}]\] 
 has infinite order, a contradiction. 
\end{proof}

We can now complete the proof of Theorem \ref{main}.

\begin{proof}{\emph{of Theorem \ref{main}.}} Recall that we wish to prove that $G' $ is locally finite 
whenever $G\in\X$. By Lemma \ref{G'} for every group $G\in\X$ all Sylow subgroups of $G'$ are periodic. 
The theorem will follow once it is shown that
\begin{itemize}
\item[(*)]\label{Y} If $G\in\Y$ and all Sylow subgroups of $G'$ are periodic, then  $G' $ is \lc. 
\end{itemize}

If $2 \notin \pi(G')$,  then $G'$ is prosoluble \cite{FT} and so is $G$. Thus, by Proposition \ref{pi-prosol} 
$ph(G')$, and hence also $ph(G)$, are finite. In this case by Proposition \ref{prosol+h=>lc} $G'$ is \lc, as required. 

So, assume that $2 \in \pi(G')$.   
By  Theorem \ref{wil-series},  $G'$ has a finite series of closed
characteristic subgroups in which each factor either is pro-($2$-soluble)
or is isomorphic to a Cartesian product of non-abelian finite simple
groups.  Since  finite $2'$-groups are soluble, $G$ has a finite series of closed characteristic subgroups 
\[1=G_{s+1} < G_s  < \ldots < G_1 < G_0=G\] 
  in which each factor either is prosoluble 
or is isomorphic to a Cartesian product of non-abelian finite simple
groups; moreover we can assume $G_1 \le G'$.  

We will prove (*) by induction on the minimal number $n$ of non prosoluble 
 factors in this series. If $n=0$, then $G$ is prosoluble, and, as above, the theorem follows from Propositions \ref{pi-prosol} and \ref{prosol+h=>lc}. 

So let $n \ge 1$. If the last term $G_s$ is isomorphic to a Cartesian product of non-abelian finite simple groups, then  by Lemma \ref{Y-nonab} it is locally finite. Since the corresponding series of $G/G_s$ has $n-1$ non-prosoluble quotients, the result follows by induction.  

Therefore we assume that $G_s$ is prosoluble. Since $G_s \le G'$, all Sylow subgroups of $G_s$ are periodic. Hence, by Propositions  \ref{pi-prosol} and  \ref{prosol+h=>lc}, the derived subgroup $G'_s$ is \lc. Passing to the quotient $G/G'_s$, we can assume that $G_s$ abelian. Then Lemma \ref{Nabel} states that there exists an open normal subgroup $H$ of $G$ such that $[G_s,H]$ is \lc. Passing to the quotient $G/[G_s,H]$,
 we can assume that  $G_s\le Z(H)$.

Since $G_{s-1}/G_s$ is a Cartesian product of non-abelian finite simple groups, by Lemma \ref{Y-nonab} $G_{s-1}/G_s$ has finite exponent. Set $K=H \cap G_{s-1}$. Then $Z(K)\ge Z(H)\cap G_{s-1}\ge G_s$. It follows that  $K/Z(K)$ has finite exponent. A theorem of Mann states that if $B$ is a finite group such that $B/Z(B)$ has exponent $e$, then the exponent of $B'$ is bounded by a function of $e$ \cite{M}. Applying a profinite version of this theorem we deduce that the exponent of $K'$ is finite. Thus $K'$ is \lc. Passing to the quotient $G/K'$, we can assume that $K$ is abelian. 

Since $|G:H|$ is finite,  $K$ has finite index in $G_{s-1}$. By Lemma \ref{abelian} the subgroup $[K,G_{s-1}]$ is \lc. Passing to the quotient $G/[K,G_{s-1}]$ we assume that $K \le Z(G_{s-1})$. Again, by the profinite   version of  Mann's theorem, we conclude that $G'_{s-1}$ is \lc. Then we apply the inductive hypothesis to $G/G'_{s-1}$ and conclude that $G'$ is \lc. The proof is complete.
\end{proof}




\end{document}